\documentclass[12pt,twoside,a4paper]{amsart}
\usepackage[textwidth=15cm,textheight=22cm,centering]{geometry}
\usepackage[T1]{fontenc}
\usepackage{lmodern}
\usepackage{amsmath,amssymb,amsthm,mathtools}
\usepackage[hidelinks]{hyperref}
\hypersetup{pdftitle={Conditional Oka Complex Structures of the Six-Sphere},
 pdfauthor={Zhangchi Chen}}

\newtheorem{maintheorem}{Theorem}

\newtheorem{dgkxtheorem}{Theorem}

\newtheorem{theorem}{Theorem}[section]
\newtheorem{proposition}[theorem]{Proposition}
\newtheorem{lemma}[theorem]{Lemma}
\theoremstyle{definition}
\newtheorem{definition}[theorem]{Definition}
\newtheorem{question}[theorem]{Question}
\newtheorem{condition}{Condition}

\newcommand{\C}{\mathbb C}
\newcommand{\R}{\mathbb R}
\newcommand{\Z}{\mathbb Z}
\newcommand{\PP}{\mathbb P}
\newcommand{\OO}{\mathcal O}
\newcommand{\Tot}{\operatorname{Tot}}
\newcommand{\Sing}{\operatorname{Sing}}
\newcommand{\Et}{\overline{\mathcal E}}

\title{Conditional Oka Complex Structures of the Six-Sphere}
\author{Zhangchi Chen}
\address{Zhangchi Chen,
School of Mathematical Sciences,
Key Laboratory of MEA (Ministry of Education),
and Shanghai Key Laboratory of PMMP,
East China Normal University,
Shanghai 200241,
China}
\email{zcchen@math.ecnu.edu.cn}
\subjclass[2020]{Primary 32Q56; Secondary 32Q65, 32M05}
\keywords{Oka manifold, six-sphere, torus fibration}
\date{}

\begin{document}
\begin{abstract}
Alp\"oge's construction produces a period family of compact complex threefolds fibred over $\PP^1$,
with complex two-tori as general fibres.
We give a direct proof that every admissible member is Oka,
using only the complex-analytic construction.
Du--Guo--Kusakabe--Xie independently established the Oka property and determined the exact biholomorphism classes in this family.
Assuming Alp\"oge's identification of the underlying smooth manifolds with $S^6$,
these results yield uncountably many pairwise nonbiholomorphic Oka complex structures on the six-sphere.
\end{abstract}
\maketitle

\section{Introduction and main results}
\label{sec:intro}

The existence of a complex structure on the six-sphere has long attracted attention in complex geometry.
Alp\"oge has recently proposed a constructive solution by completing a family of complex two-tori over $\PP^1$ \cite{Alpoge}.
His construction leads naturally to questions about the complex-analytic properties of the resulting compact threefolds.

Oka manifolds are natural targets for entire maps $\C^n\to Y$,
and more generally for holomorphic maps from Stein manifolds:
their mapping theory allows deformation,
approximation and interpolation subject to the corresponding topological constraints \cite{ForstnericBook,ForstnericICM}.
Finding new compact Oka manifolds is an important problem in the theory.
These considerations lead to the question:
does the six-sphere admit an Oka complex structure?

Oka theory, named after Kiyoshi Oka,
originates in his 1939 work on the second Cousin problem:
on a domain of holomorphy in $\C^n$,
topological triviality of a holomorphic line bundle implies holomorphic triviality \cite{Oka1939}.
Grauert extended this principle in 1958 to principal bundles with complex Lie structure groups over Stein spaces,
proving that their holomorphic and topological classifications coincide \cite{Grauert1958}.
This is the classical Oka--Grauert principle.
The notion of an Oka manifold was already implicit in Grauert's work.
Gromov introduced holomorphic sprays and ellipticity as a sufficient geometric condition for the Oka principle,
and used the notation $\mathrm{Ell}_{\infty}$ for the Oka property \cite{Gromov,ForstnericICM}.
Forstneri\v c introduced the convex approximation property (CAP) in 2006 to unify the various Oka properties \cite{ForstnericCAP}.
In 2009 he proved that CAP also implies the parametric Oka properties and introduced the name \emph{Oka manifold} \cite[Theorem~1.1 and Definition~1.2]{ForstnericOka2009};
see \cite[Section~5.4, especially Theorem~5.4.4]{ForstnericBook} for the equivalence.
Kusakabe's 2021 characterization by the convex form of Gromov's condition $\mathrm{Ell}_1$ underlies his Zariski localization theorem \cite{Kusakabe};
this condition is now called \emph{relative convex ellipticity} \cite[Definition~1.5 and Theorems~1.6--1.7]{ForstnericICM}.

We give a direct proof that every admissible member of Alp\"oge's period family is Oka.

Let $Y_{c_0}$ denote the member of Alp\"oge's Construction~6.1 with the fixed logarithmic twists
\[
v_1=\widehat\gamma+2\widehat u-4\widehat w,\qquad v_2=-\widehat\gamma-3\widehat u+3\widehat w,
\]
and with untranslated cusp gluing.
Write $\beta=\beta^{\rm part}+c_0$,
where $\beta^{\rm part}$ is a fixed particular solution of the period transformation laws and cusp condition,
normalized by \eqref{eq:normbeta};
the superscript abbreviates ``particular'' \cite[Theorem~3.4(iii)]{Alpoge}.
Following \cite[Remark~6.4]{Alpoge} and \cite[Proposition~2.3 and the subsequent definition]{DGKX},
we call $c_0$ \emph{admissible} when the inequality $(\beta3)$ of \cite[Definition~3.1]{Alpoge} holds.
Equivalently,
in the notation of Section~\ref{sec:construction},
\begin{equation}
\label{eq:admissible}
D(\beta^{\rm part}+c_0):= \operatorname{Im}(\beta^{\rm part}+c_0) -\frac{6(\operatorname{Im}\mu)^2}{\operatorname{Im}\tau}<0 \quad\text{on }\mathfrak h_z.
\end{equation}
After the normalization \eqref{eq:normbeta},
this is equivalent to $\operatorname{Im}c_0<0$.
The source proves that every such choice yields a compact complex threefold \cite[Theorem~6.2 and Remark~6.4]{Alpoge}.

\begin{maintheorem}
\label{thm:main}
Let $Y=Y_{-i}$ be the member fixed in \eqref{eq:normbeta}--\eqref{eq:c0}.
Then $Y$ is Oka.
\end{maintheorem}

The same Oka argument applies to each admissible period parameter.

\begin{maintheorem}
\label{thm:family}
For every $c_0$ with $\operatorname{Im}c_0<0$,
equivalently satisfying \eqref{eq:admissible},
the complex threefold $Y_{c_0}$ with the fixed twists and cusp gluing above is Oka.
\end{maintheorem}

Du, Guo, Kusakabe and Xie independently proved the Oka assertion and determined the exact biholomorphism classes in this family.
We collect their results in the following theorem \cite[Theorems~1.1 and~1.4, Corollary~2.28]{DGKX}.
To compare notation,
choose their reference solution $\beta_0$ to be our normalized $\beta^{\rm part}$.
Their period function $\beta_0+a$ then equals $\beta^{\rm part}+c_0$ with $a=c_0$,
and their admissible half-plane is $\operatorname{Im}c_0<0$ \cite[Propositions~2.2--2.3]{DGKX}.
With the fixed twists and cusp gluing above,
their family is identified with ours by \cite[Lemma~2.29]{DGKX}.

\begin{dgkxtheorem}
\label{thm:DGKX}
For every $c_0$ with $\operatorname{Im}c_0<0$,
the manifold $Y_{c_0}$ is Oka,
and all these manifolds have the same smooth type.
For any two admissible parameters,
\begin{equation}
\label{eq:classification}
Y_{c_0}\simeq Y_{c_0'}
\quad\Longleftrightarrow\quad
c_0'-c_0\in2\Z.
\end{equation}
In particular, $Y_{-it}$, $t>0$, are pairwise nonbiholomorphic.
\end{dgkxtheorem}

Du--Guo--Kusakabe--Xie further prove that the family projection over each relatively compact parameter disc is an Oka map \cite[Theorem~1.1]{DGKX}.
We use their results without reproving them.

Theorems~\ref{thm:main} and~\ref{thm:family} do not depend on the following condition,
which we use only for the conclusions concerning $S^6$.

\begin{condition}[Alp\"oge; Du--Guo--Kusakabe--Xie]
\label{cond:sphere}
For every $c_0$ with $\operatorname{Im}c_0<0$,
the underlying smooth manifold of $Y_{c_0}$ is diffeomorphic to $S^6$.
\end{condition}

Alp\"oge's Theorem~8.1 asserts this identification,
and his Remark~6.4 specifies that the construction and proofs apply to every admissible parameter \cite{Alpoge}.
Du--Guo--Kusakabe--Xie construct a proper holomorphic submersion over each relatively compact parameter disc,
with the threefolds $Y_{c_0}$ as fibres,
and deduce constancy of their smooth type by Ehresmann's theorem \cite[Theorem~2.24 and Corollary~2.28]{DGKX}.
This argument does not use the identification with $S^6$.
After identifying Alp\"oge's construction with a member of their family \cite[Lemma~2.29]{DGKX},
they invoke his Theorem~8.1 to identify the common smooth type with $S^6$ \cite[Corollary~2.30]{DGKX}.
Their proof of the sphere identification therefore retains this input from Alp\"oge.

Under Condition~\ref{cond:sphere},
Theorem~\ref{thm:DGKX} gives uncountably many pairwise nonbiholomorphic Oka complex structures on $S^6$,
and all currently known constructions of complex structures on $S^6$ give Oka manifolds.
Within this family,
the biholomorphically invariant Oka property is constant on the fixed diffeomorphism type $S^6$ despite the variation of biholomorphism class.
In this limited sense,
the result may be viewed as an analogue of an Oka principle for these compact complex six-spheres.
Such constancy already fails in the noncompact setting:
$\C$ and the unit disc are diffeomorphic,
but only the former is Oka \cite[Section~1]{ForstnericICM}.

Our proof combines sprays constructed from complete holomorphic flows on explicit surface charts,
unramified cyclic covers and Kusakabe's localization theorem,
with the precise tools recalled in Section~\ref{sec:tools}.
A cyclic base change and an invariant-lattice cover reduce the one-row pieces to $\C^*$-bundles over explicit flexible surfaces.
The remaining two-row pieces are small resolutions of line-bundle-valued nodal suspensions.
We establish their Oka property in Sections~\ref{sec:scalar} and~\ref{sec:twisted},
then use Zariski localization to assemble the threefold.
Line-bundle twisting is handled by pullback to a principal $\C^*$-bundle.
The fixed member $Y$ is not homogeneous,
since Alp\"oge computes $h^0(Y,T_Y)=1$ \cite[Theorem~9.1(v)]{Alpoge}.

The author gratefully acknowledges the contributions on which this work rests:
Alp\"oge's construction and the smooth identification stated in \cite{Alpoge},
Gromov's spray method \cite{Gromov},
Forstneri\v c's unification of Oka theory and the permanence results used here \cite{ForstnericCAP,ForstnericOka2009,ForstnericBook},
and Kusakabe's localization theorem \cite{Kusakabe}.
Xie's work on cubic hypersurface complements develops related techniques using unramified cyclic covers and complete fibrewise Hamiltonian flows with line-bundle time parameters \cite[Sections~1 and~3]{XieCubic}.
In particular,
Du--Guo--Kusakabe--Xie have completely resolved the question about biholomorphism classes raised in Remark~1.1 and Section~10 of the \href{https://arxiv.org/abs/2609.26706v1}{first version} of this paper.
The results collected in Theorem~\ref{thm:DGKX} are entirely credited to them.
Their manuscript records that they obtained the independent Oka result in August 2026,
before our first version;
the author learned of their work from Song-Yan Xie shortly after the arXiv announcement.

We conclude in Section~\ref{sec:questions} with questions about Gromov ellipticity of $Y_{c_0}$,
the existence of complex structures on $S^6$ outside this family up to biholomorphism,
and the Oka property of any such structures.

\section{The threefold and its period parameter}
\label{sec:construction}

Let $B=\PP^1$,
with affine coordinate $t$,
and put
\[
p_1=0,\qquad p_2=1,\qquad p_0=\infty,\qquad t_c=t^{-1},\qquad B^\circ=B\setminus\{p_0,p_1,p_2\}.
\]
The orbifold $(B;3,4,\infty)$ is hyperbolic.
Its orbifold universal cover is a copy $\mathfrak h_z$ of the upper half-plane,
with deck group
\[
\Delta=\langle g_1,g_2\mid g_1^3=g_2^4=1\rangle,
\qquad g_0=(g_1g_2)^{-1}.
\]
We write $\pi:\mathfrak h_z\to B\setminus\{p_0\}$ for the quotient map,
choose fixed points $z_j$ of $g_j$,
and choose linearizing coordinates $s_j$ with
\[
s_j(z_j)=0,\qquad s_j(g_jz)=e^{-2\pi i/m_j}s_j(z),
\qquad (m_1,m_2)=(3,4).
\]
These are the standard consequences of the uniformization theorem for the triangle orbifold \cite{Beardon}.

\subsection{Lattice data and monodromy}

Let
\[
\Lambda=\langle\widehat\gamma,\widehat u,
\widehat w,\widehat\delta\rangle\simeq\Z^4.
\]
Let $(\gamma,u,w,\delta)$ be the dual basis of $\operatorname{Hom}(\Lambda,\Z)$.
In the ordered basis $(\widehat\gamma,\widehat u,\widehat w,\widehat\delta)$ of $\Lambda$,
the source monodromies are
\[
A_1= \begin{pmatrix} 1&0&0&0\\ 6&0&1&0\\ -6&-1&-1&0\\ -2&1&0&1 \end{pmatrix},\qquad A_2= \begin{pmatrix} 1&0&0&0\\ 0&0&-1&0\\ -6&1&0&0\\ 3&0&1&1 \end{pmatrix},
\]
and
\[
M_0= \begin{pmatrix} 1&0&0&0\\ 0&1&0&0\\ 0&1&1&0\\ -1&0&0&1 \end{pmatrix}.
\]
Direct multiplication gives
\[
A_1^3=A_2^4=I,\qquad A_1A_2M_0=I.
\]
Thus there is a representation $\rho_\Lambda:\Delta\to\operatorname{GL}(\Lambda)$ with $\rho_\Lambda(g_j)=A_j$ and $\rho_\Lambda(g_0)=M_0$.
Put $M_g=\rho_\Lambda(g)^{-1}$;
then $M_{gh}=M_hM_g$.
Put
\[
\Lambda_{\rm tor}=\langle\widehat w,\widehat\delta\rangle,
\qquad \overline\Lambda=\Lambda/\Lambda_{\rm tor},
\qquad K=\ker\gamma =\langle\widehat u,\widehat w,\widehat\delta\rangle.
\]
Here $\gamma:\Lambda\to\Z$ is the coordinate functional with $\gamma(\widehat\gamma)=1$ and $\gamma(\widehat u)=\gamma(\widehat w) =\gamma(\widehat\delta)=0$.
The formulas for $M_0$ give,
without any analytic input,
\[
\Lambda_{\rm tor}=\ker(M_0-I)=\operatorname{im}(M_0-I)
\]
and the induced unimodular map
\begin{equation}
B_0:\overline\Lambda\longrightarrow\Lambda_{\rm tor},
\qquad B_0\overline{\widehat\gamma}=-\widehat\delta,
\quad B_0\overline{\widehat u}=\widehat w.
\label{eq:B0}
\end{equation}
The same matrices show that $K$ is invariant under $A_1,A_2,M_0$,
that every global monodromy fixes $\widehat\delta$,
and that $K$ is the smallest monodromy-invariant subgroup containing $\Lambda_{\rm tor}$.

We fix
\begin{equation}
v_1=\widehat\gamma+2\widehat u-4\widehat w,
\qquad v_2=-\widehat\gamma-3\widehat u+3\widehat w.
\label{eq:twists}
\end{equation}
Then
\[
A_1v_1=v_1,\qquad A_2v_2=v_2,
\qquad \gamma(v_1)=1,\quad \gamma(v_2)=-1.
\]

\subsection{The period functions}

Put $\rho=e^{i\pi/3}$.
We use the modular function normalized by $j(\rho)=0$,
$j(i)=1728$,
and $j(\tau)=q^{-1}(1+O(q))$,
$q=e^{2\pi i\tau}$.
The standard facts about $j,E_4,E_6$,
and $\Delta_{\rm mod}$ used below are recalled in \cite{SerreArithmetic}.
We record the period data of \cite[Theorem~3.4]{Alpoge},
including their normalization,
to fix the coordinates used in the Oka argument.

\begin{proposition}
\label{prop:period}
There are unique holomorphic functions $\tau:\mathfrak h_z\to\mathfrak h$ and $\mu:\mathfrak h_z\to\C$,
and a nonempty affine family of holomorphic functions $\beta=\beta^{\rm part}+c_0$,
$c_0\in\C$,
satisfying
\begin{align}
\tau(g_1z)&=\frac{\tau(z)-1}{\tau(z)},& \tau(g_2z)&=-\frac1{\tau(z)},
\label{eq:tau-laws}
\\ \mu(g_1z)&=\frac{1-\mu(z)}{\tau(z)},& \mu(g_2z)&=1+\frac{\mu(z)}{\tau(z)},
\label{eq:mu-laws}
\\ \beta(g_1z)&=\beta(z)+2-\frac{6(1-\mu(z))^2}{\tau(z)},& \beta(g_2z)&=\beta(z)-3-\frac{6\mu(z)^2}{\tau(z)}.
\label{eq:beta-laws}
\end{align}
They have
\[
\tau(z_1)=\rho,\quad \tau(z_2)=i,\quad \mu(z_1)=\frac{2-\rho}{3},\quad \mu(z_2)=\frac{1-i}{2}.
\]
On a distinguished cusp lift,
\begin{equation}
e^{2\pi i\tau}=t_cu_1(t_c),\qquad u_1(0)\ne0,
\label{eq:qexp}
\end{equation}
where $u_1$ is holomorphic,
while $\mu$ and $\beta+\tau$ descend and extend holomorphically across $t_c=0$.
\end{proposition}

\begin{proof}
Set $F=1728(t\circ\pi)$ and remove the inverse images of $0,1728$.
The restriction of $j$ away from the elliptic orbits is a regular covering with deck group $\operatorname{PSL}_2(\Z)$.
A meridian under $F$ is a cube at $p_1$ and a fourth power at $p_2$;
its image in the deck group is therefore trivial,
since the corresponding modular stabilizers have orders three and two.
The covering criterion gives a lift $j(\tau)=F$.
Composing with a biholomorphism $\mathfrak h\simeq\mathbb D$ and applying the removable singularity theorem extends $\tau$ across the elliptic orbits.
Their local degrees force $\tau-\rho$ to have order one at $z_1$ and $\tau-i$ order two at $z_2$.
Normalize the first stabilizer to act by
\[
\mathsf S_1=\begin{pmatrix}1&-1\\1&0\end{pmatrix},\qquad \mathsf S=\begin{pmatrix}0&-1\\1&0\end{pmatrix},\qquad \mathsf T=\begin{pmatrix}1&1\\0&1\end{pmatrix},
\]
where the matrices act by fractional linear transformations.
The second stabilizer has order two in $\operatorname{PSL}_2(\Z)$.
The image $\mathsf P$ of $g_0$ is nontrivial:
if $\mathsf P=1$,
then $\tau$ would give a single-valued logarithm of the cusp parameter,
contradicting its positive winding number.
For large $R$,
every component of $\{|j|>R\}$ is trapped in a horoball at one modular cusp.
Since $\mathsf P$ preserves the component containing the cusp lift,
it fixes that cusp and is therefore parabolic.
The relation $g_1g_2g_0=1$,
together with the order-two condition on the second stabilizer,
leaves,
up to conjugation by a power of $\mathsf S_1$,
exactly $\mathsf P=\mathsf T^{-1}$ and the second stabilizer $\mathsf S$.
Thus \eqref{eq:tau-laws} holds and $\tau(g_0z)=\tau(z)-1$,
since $\mathsf S_1\mathsf S=\mathsf T$.
The centralizer of $\mathsf S_1,\mathsf S$ in $\operatorname{PSL}_2(\Z)$ is trivial,
proving uniqueness.
At the cusp,
$q=e^{2\pi i\tau}$ is single valued and bounded.
The simple pole of $j$ gives \eqref{eq:qexp}.

For $\mu$,
define the factor of automorphy
\[
\widetilde j(g_1,z)=-\tau(z),\qquad \widetilde j(g_2,z)=\tau(z).
\]
The homogeneous law $\nu(gz)=\nu(z)/\widetilde j(g,z)$,
with cusp regularity,
defines a line bundle $\mathcal L$ on $B$.
On the simply connected cover choose a square root of $E_6\circ\tau$.
The function
\[
\Phi=\left(E_4^2E_6^{1/2}/\Delta_{\rm mod}\right)\circ\tau
\]
has modular weight $-1$,
up to a sign character from the square root.
Iteration around the order-three point forces that sign to be positive there;
at the order-four point the local coordinate and the simple zero of $E_6^{1/2}\circ\tau$ force the same sign.
Hence $\Phi$ obeys the homogeneous law.
It has zeros of orders two and one over $p_1,p_2$,
and equals $t_c^{-1}$ times a unit at $p_0$.
Thus an equivariant section divided by $\Phi$ descends holomorphically at the elliptic points and vanishes at the cusp,
proving $\mathcal L\simeq\OO_B(-p_0)\simeq\OO_{\PP^1}(-1)$.
The affine local solutions of \eqref{eq:mu-laws} form an $\mathcal L$-torsor.
Local solutions are obtained by choosing a sheet away from the marked points,
by $\mu=0$ at the cusp,
and by $(2-\tau)/3$,
respectively $(1-\tau)/2$,
near $p_1,p_2$.
Since $H^1(B,\OO(-1))=H^0(B,\OO(-1))=0$,
there is exactly one global $\mu$.
Evaluation at the fixed points gives the displayed values.

Finally put
\[
\varphi_1=2-\frac{6(1-\mu)^2}{\tau},\qquad \varphi_2=-3-\frac{6\mu^2}{\tau}.
\]
Substitution using \eqref{eq:tau-laws}--\eqref{eq:mu-laws} gives
\[
\sum_{k=0}^{2}\varphi_1\circ g_1^k=0,\qquad \sum_{k=0}^{3}\varphi_2\circ g_2^k=0.
\]
Hence \eqref{eq:beta-laws} is locally solvable:
at $p_j$ take $m_j^{-1}\sum_{k=0}^{m_j-1}k(\varphi_j\circ g_j^k)$,
at the cusp take $-\tau$,
and elsewhere prescribe a value on one sheet.
The local solutions form an $\OO_B$-torsor.
Since $H^1(B,\OO_B)=0$,
global solutions exist,
and $H^0(B,\OO_B)=\C$ shows that they differ by one constant.
A direct composition gives
\[
(\tau,\mu,\beta)\circ(g_1g_2)=(\tau+1,\mu,\beta-1),
\]
so $\mu$ and $\beta+\tau$ are cusp invariant;
the imposed local regularity gives their holomorphic extensions.
\end{proof}

\subsection{Normalization, nondegeneracy, and cusp periods}

For a particular solution $\beta^{\rm part}$,
put
\[
D(\beta)=\operatorname{Im}\beta -\frac{6(\operatorname{Im}\mu)^2}{\operatorname{Im}\tau}.
\]

\begin{proposition}
\label{prop:periodlattice}
The function $D(\beta^{\rm part})$ is $\Delta$-invariant,
bounded above,
and tends to $-\infty$ at the cusp.
There is a unique normalization
\begin{equation}
\operatorname{Re}\beta^{\rm part}(z_2)=0,
\qquad \sup_{\mathfrak h_z}D(\beta^{\rm part})=0.
\label{eq:normbeta}
\end{equation}
With this normalization the admissible constants are exactly those with $\operatorname{Im}c_0<0$.
For this normalization set
\begin{equation}
\beta^*=\beta^{\rm part}-i.
\label{eq:c0}
\end{equation}
Then $D(\beta^*)\le-1$,
and
\begin{equation}
\Pi(z)= \begin{pmatrix} 6\mu(z)&\tau(z)&1&0\\ \beta^*(z)&\mu(z)&0&1 \end{pmatrix}
\label{eq:Pi}
\end{equation}
maps $\Lambda\otimes\R$ isomorphically onto $\C^2$.
Moreover there are invertible holomorphic matrices $R_g$ satisfying
\[
\Pi(gz)=R_g(z)\Pi(z)M_g.
\]
At the cusp,
\begin{equation}
\Pi=[sB_0+C(t_c)\mid I],\qquad e^{2\pi i s}=t_c,
\label{eq:cuspperiod}
\end{equation}
with $C$ holomorphic at $t_c=0$.
For every admissible $c_0$,
the lattice,
equivariance,
and cusp-form assertions hold with $\beta^*$ replaced by $\beta_{c_0}=\beta^{\rm part}+c_0$:
the period matrix $\Pi_{c_0}$ defines a family of compact complex two-tori,
is equivariant under the same monodromy,
and has the same cusp form with $C$ replaced by a holomorphic $C_{c_0}$.
\end{proposition}

\begin{proof}
For the generators,
direct multiplication of the matrices in the preceding subsection gives the equivariance with
\[
R_{g_1}=\begin{pmatrix}-1/\tau&0\\(1-\mu)/\tau&1\end{pmatrix},
\qquad R_{g_2}=\begin{pmatrix}1/\tau&0\\-\mu/\tau&1\end{pmatrix}.
\]
Both are invertible;
the cocycle rule extends them to $\Delta$.
Viewing $\Pi$ as a real map and writing $Z=\left(\begin{smallmatrix}6\mu&\tau\\\beta&\mu\end{smallmatrix}\right)$,
one obtains
\begin{equation}
\det_{\R}\Pi=-\det\operatorname{Im}Z =\operatorname{Im}\tau\,D(\beta).
\label{eq:real-det}
\end{equation}
Taking real determinants in the two equivariance identities and using
\[
|\det R_{g_j}|^2=|\tau|^{-2},\qquad \operatorname{Im}\tau(g_jz)=\frac{\operatorname{Im}\tau}{|\tau|^2},
\]
shows that $D$ is invariant.
If $b=\beta+\tau$,
then at the cusp
\[
D\le\operatorname{Im}b-\operatorname{Im}\tau\longrightarrow-\infty.
\]
It therefore has finite supremum on the quotient.
Independent real and imaginary translations of $\beta^{\rm part}$ give the unique normalization \eqref{eq:normbeta}.
The supremum is attained because $D$ is continuous on the punctured compact base and tends to $-\infty$ at the cusp.
Since $D(\beta^{\rm part}+c_0)=D(\beta^{\rm part})+ \operatorname{Im}c_0$,
the strict inequality holds everywhere exactly when $\operatorname{Im}c_0<0$;
subtracting $i$ gives $D(\beta^*)\le-1$.
Equation \eqref{eq:real-det} now proves that $\Pi(z)\Lambda$ is a full lattice in $\C^2$ for every $z$.

For any admissible $c_0$,
the same determinant calculation has $D(\beta_{c_0})<0$.
Adding a constant to $\beta$ changes neither the transformation laws nor the extension of $\beta+\tau$ at the cusp;
the equivariance and cusp calculations therefore apply with $\Pi_{c_0}$ and $C_{c_0}$ as well.

Choose a branch $h=(2\pi i)^{-1}\log u_1$ and put $s=\tau-h$.
Since $\mu$ and $b=\beta^*+\tau$ are holomorphic in $t_c$,
\[
C(t_c)=\begin{pmatrix}6\mu(t_c)&h(t_c)\\ b(t_c)-h(t_c)&\mu(t_c)\end{pmatrix}
\]
is holomorphic at zero,
and substitution gives \eqref{eq:cuspperiod}.
\end{proof}

The lattice family $\mathcal T=(\mathfrak h_z\times\C^2)/\Lambda$,
with translation by $\Pi(z)\lambda$,
is therefore a holomorphic family of compact two-tori.
The equivariance gives a $\Delta$-action.
Away from the two elliptic orbits it is free and properly discontinuous,
so
\[
J=\mathcal T|_{\pi^{-1}(B^\circ)}/\Delta\longrightarrow B^\circ
\]
is a proper holomorphic torus submersion.

\begin{proposition}
\label{prop:weierstrass}
The elliptic quotient with period lattice $\Z+\tau\Z$ has relatively minimal Weierstrass model
\begin{equation}
y^2=4x^3-3t^3(t-1)x-t^4(t-1)^2,
\label{eq:weierstrass}
\end{equation}
and the displayed equation has the section $P=(0,it^2(t-1))$.
After the cubic base change
\[
t=\frac{s^3}{s^3-1},\qquad x=\frac{s^4X}{(s^3-1)^2},\qquad y=\frac{s^6Y}{(s^3-1)^3},
\]
it becomes $Y^2=4X^3-3sX-1$,
and $P=(0,i)$.
\end{proposition}

\begin{proof}
For \eqref{eq:weierstrass},
\[
\Delta_W=27t^8(t-1)^3,
\qquad j=1728\frac{(3t^3(t-1))^3}{\Delta_W}=1728t.
\]
The vanishing orders give Kodaira fibres $IV^*,III,I_1$ at $0,1,\infty$,
respectively \cite{BHPV}.
Their local monodromies agree with \eqref{eq:tau-laws};
uniqueness of the normalized modular lift and of the relatively minimal extension identifies the elliptic family.
The point $P$ satisfies the equation,
and the base-change assertion is direct substitution.
\end{proof}

\subsection{The toroidal cusp filling}

Identify $\overline\Lambda$ and $\Lambda_{\rm tor}$ with $\Z^2$ in the bases already chosen.
Let $\mathcal T_{A_2}$ be the triangulation of $\R^2$ with vertices $\Z^2$,
cut out by $y_1\in\Z$,
$y_2\in\Z$,
$y_1+y_2\in\Z$.
In $N'=\Z^2\oplus\Z$,
let $\mathcal F$ consist of the cones over its cells at height one,
together with the zero cone,
and let $\mathcal Y$ be the associated toric threefold.
The two types of maximal cone have determinants $+1$ and $-1$;
hence $\mathcal Y$ is smooth \cite{Fulton}.
On each maximal chart $U_\sigma\simeq\C^3$ there are coordinates with
\begin{equation}
t_c=z_0z_1z_2.
\label{eq:toriclocal}
\end{equation}
Thus $t_c^{-1}(0)$ is a reduced normal-crossings divisor.

Translation of the height-one fan by $B_0\bar\lambda$ induces a toric automorphism $\Phi_{\bar\lambda}$.
For $|t_c|<r$,
define
\begin{equation}
\Psi_{\bar\lambda}(x,t_c)= \bigl(c_{\bar\lambda}(t_c)t_c^{B_0\bar\lambda}x,t_c\bigr),
\qquad c_{\bar\lambda}=e^{2\pi iC(t_c)\bar\lambda}.
\label{eq:Psi}
\end{equation}
Equivalently,
on all of $\mathcal Y$,
first apply $\Phi_{\bar\lambda}$ and then multiply by $(c_{\bar\lambda}(t_c),1)$ in the acting torus.

The cusp filling is the one proved in \cite[Theorem~4.5]{Alpoge};
the following calculation records the properness estimate used here.
\begin{proposition}
\label{prop:cuspfill}
For sufficiently small $\epsilon>0$,
the action $\bar\lambda\mapsto\Psi_{\bar\lambda}$ on $\mathcal Y_\epsilon=\{|t_c|<\epsilon\}$ is free and properly discontinuous.
Its quotient $N_0=\mathcal Y_\epsilon/\overline\Lambda$ is a connected Hausdorff complex threefold,
proper over the cusp disc,
with reduced normal-crossings central fibre $W$.
The map
\begin{equation}
E_0^{\exp}(z,\zeta)= (e^{2\pi i\zeta_1},e^{2\pi i\zeta_2},e^{2\pi is(z)})
\label{eq:Eexp}
\end{equation}
induces a biholomorphism $G_0:J|_{0<|t_c|<\epsilon}\to N_0\setminus W$.
\end{proposition}

\begin{proof}
Put $\mathcal R(t_c)=-2\pi\operatorname{Im}C(t_c)$.
Choose $\epsilon$ so small that
\[
|\log\epsilon|\ge2\sup_{|t_c|\le\epsilon}\|\mathcal R(t_c)\|,
\qquad B_t=B_0+\frac{\mathcal R(t_c)}{\log|t_c|}
\]
satisfies
\begin{equation}
\|B_t-B_0\|\le\frac12,\qquad |B_t\bar\lambda|\ge\frac12|\bar\lambda|.
\label{eq:Bt}
\end{equation}
On the dense torus,
a fixed point would give $\log|t_c|B_0\bar\lambda+\mathcal R(t_c)\bar\lambda=0$,
contrary to \eqref{eq:Bt}.
On the central fibre,
$\Psi_{\bar\lambda}$ translates the cell indexing a toric orbit by $B_0\bar\lambda$;
a nonzero translation cannot preserve a bounded cell.
The action is free.

For a torus point set $y=\log|x|/\log|t_c|$.
If a translate of one fixed toric chart meets another,
the corresponding rescaled positions satisfy
\[
y(\Psi_{\bar\lambda}p)=y(p)+B_t\bar\lambda.
\]
The enlarged triangles belonging to two fixed charts are bounded,
while \eqref{eq:Bt} bounds $|\bar\lambda|$.
Only finitely many translates can therefore meet two prescribed charts,
and compact sets meet only finitely many charts.
This proves proper discontinuity.

It remains to prove properness over the disc.
Given a torus point,
choose $\bar\lambda\in\Z^2$ so that,
for $u=B_t^{-1}y(p)$,
$\|u+\bar\lambda\|_\infty\le1/2$.
Then
\[
|y(\Psi_{\bar\lambda}p)| \le\frac32\frac{\sqrt2}{2}<\sqrt2.
\]
Thus every orbit has a representative in finitely many closed toric polydiscs.
On the central fibre the same conclusion follows by translating a vertex to zero,
using $B_0\overline\Lambda=\Z^2$.
Over every smaller closed disc these polydiscs form a compact fundamental set.
Hence the quotient is Hausdorff and proper;
\eqref{eq:toriclocal} gives the asserted central fibre.

Finally \eqref{eq:Eexp} is a surjective local biholomorphism onto the dense torus over the punctured disc.
Its fibres are the orbits generated by $g_0$ and $\Lambda_{\rm tor}$,
and \eqref{eq:cuspperiod} shows that it intertwines the residual lattice action with \eqref{eq:Psi}.
It therefore descends to $G_0$.
\end{proof}

\subsection{The logarithmic fillings}

Choose small $g_j$-invariant discs $\Delta_j\ni z_j$,
with the linearizing coordinates $s_j$,
and put
\[
\mathcal T_j=(\Delta_j\times\C^2)/\Lambda.
\]
The period equivariance makes
\begin{equation}
\widetilde g_j^{\log}(z,\zeta)= \left(g_jz,R_{g_j}(z)\zeta+ \Pi(g_jz)\frac{v_j}{m_j}\right)
\label{eq:logaction}
\end{equation}
a holomorphic automorphism of $\mathcal T_j$.

The logarithmic filling is \cite[Theorem~5.4]{Alpoge} with our fixed twists.
We recall the freeness check because it is also needed for the unramified base change.
\begin{proposition}
\label{prop:logfill}
For $j=1,2$,
the automorphism $\widetilde g_j^{\log}$ generates a free group of order $m_j$.
The quotient $N_j=\mathcal T_j/\langle\widetilde g_j^{\log}\rangle$ is a complex threefold proper over $D_j=\Delta_j/\langle g_j\rangle$.
On the punctured disc,
translation by
\begin{equation}
\sigma_j(z)=\frac{\log s_j(z)}{2\pi i}\Pi(z)v_j
\label{eq:logsection}
\end{equation}
conjugates the logarithmic action to the untwisted action and induces a biholomorphism $G_j:N_j\setminus f_j^{-1}(p_j)\to J|_{D_j^*}$.
\end{proposition}

\begin{proof}
In the flat real coordinates $\mathcal T_j\simeq\Delta_j\times(\Lambda\otimes\R)/\Lambda$,
the action reads
\[
(z,x)\longmapsto(g_jz,A_jx+v_j/m_j).
\]
Since $A_jv_j=v_j$,
its $k$-th power translates by $(k/m_j)v_j$ after applying $A_j^k$;
its $m_j$-th power is translation by the lattice vector $v_j$,
hence the identity.
No smaller power is the identity because it covers $g_j^k$.

For completeness,
a finite-order affine torus map $x\mapsto Ax+b$ has a fixed point exactly when every integral $A$-invariant functional $\phi$ has $\phi(b)\in\Z$.
Indeed,
the condition is $b\in(A-I)(\Lambda\otimes\R)+\Lambda$,
and averaging over $\langle A\rangle$ identifies the annihilator lattice with the invariant functionals.
Direct calculation from $A_1,A_2$ gives,
for all relevant nontrivial powers,
invariant lattices generated by $\gamma$ and respectively
\[
\psi_1=2u+w+3\delta,\qquad \psi_2=u+w+2\delta.
\]
But
\[
\gamma(kv_1/3)=k/3\quad(k=1,2),\qquad \gamma(kv_2/4)=-k/4\quad(k=1,2,3)
\]
is never integral.
Hence the action is free.

A finite free holomorphic action has an unramified complex-manifold quotient \cite[Ch.~V]{GrauertRemmert}.
The torus family $\mathcal T_j\to\Delta_j$ is proper and $\Delta_j\to D_j$ is finite,
so the quotient map to $D_j$ is proper.
Branches of $\log s_j$ change \eqref{eq:logsection} by the period $\Pi(z)v_j$,
so $\sigma_j$ is a well-defined section.
Finally,
$s_j(g_jz)=e^{-2\pi i/m_j}s_j(z)$,
$R_{g_j}\Pi=\Pi(g_jz)A_j$,
and $A_jv_j=v_j$ give
\[
h_j\widetilde g_j^{\log}h_j^{-1}=\widetilde g_j^0,
\]
where $h_j$ is translation by $\sigma_j$.
Descent gives $G_j$.
\end{proof}

\subsection{The compact threefold}

The fillings $N_j$ and gluing maps $G_j$ are those of Propositions~\ref{prop:cuspfill} and~\ref{prop:logfill}.

\begin{definition}
\label{def:Y}
For each admissible $c_0$,
choose the period matrix $\Pi_{c_0}$,
the twists \eqref{eq:twists},
and the untranslated cusp gluing.
Let $J_{c_0}$,
$N_{j,c_0}$,
and $G_{j,c_0}$ be the corresponding torus family,
fillings,
and gluing maps described above.
Define
\begin{equation}
Y_{c_0}=\bigl(J_{c_0}\sqcup N_{0,c_0}\sqcup N_{1,c_0} \sqcup N_{2,c_0}\bigr)/\!\sim,
\label{eq:Y}
\end{equation}
where $\sim$ identifies the corresponding punctured-disc parts by $G_{0,c_0},G_{1,c_0},G_{2,c_0}$.
Thus the fixed gluing invariants are $(\ell_0,\ell_1,\ell_2)=(0,1,-1)$.
We abbreviate $Y=Y_{-i}$ for the normalized choice \eqref{eq:c0}.
\end{definition}

Fix once and for all one allowable set of auxiliary radii and round linearizing discs supplied by Propositions~\ref{prop:cuspfill} and~\ref{prop:logfill}.
The proof below is uniform in that choice.
Two harmless ambiguities disappear directly from the formulas:
changing a branch in \eqref{eq:cuspperiod} adds an integral multiple of $B_0$ and leaves \eqref{eq:Psi} unchanged,
while changing $\log s_j$ adds a lattice translation and leaves the induced map on the torus quotient unchanged.
No independence assertion about auxiliary choices is needed for the existence or Oka theorem for this fixed $Y$.

The compactness statement is also \cite[Theorem~6.2]{Alpoge};
we include the gluing argument to make clear that the period parameter introduces no new topological hypothesis.
\begin{proposition}
\label{prop:compact}
For every admissible $c_0$,
the space $Y_{c_0}$ is a compact connected Hausdorff complex manifold of dimension three.
\end{proposition}

\begin{proof}
Fix an admissible $c_0$,
abbreviate the corresponding pieces by $J,N_0,N_1,N_2$ and their gluing maps by $G_0,G_1,G_2$,
and take the three filling discs pairwise disjoint.
Every nonsingleton equivalence class in \eqref{eq:Y} consists of one point in $J$ and one point in exactly one $N_j$.
Since each $G_j$ is a biholomorphism between open sets,
saturation preserves openness.
The four summands therefore embed as open sets in the quotient,
and their given complex structures form a three-dimensional complex atlas.

Points over distinct base points are separated by inverse images of disjoint base neighbourhoods.
Points over the same base point lie together in one of the Hausdorff charts $J,N_0,N_1,N_2$;
hence $Y$ is Hausdorff.
The map $f:Y\to B$ induced by the four proper local maps is proper because properness is local on the base for locally compact Hausdorff spaces.
Since $B$ is compact,
$Y=f^{-1}(B)$ is compact.
Finally,
all four pieces are connected and every filling meets $J$,
so $Y$ is connected.
\end{proof}

Write $\Sigma_j=(f^{-1}(p_j))_{\mathrm{red}}$ for the two reduced multiple fibres,
$j=1,2$,
and $W=f^{-1}(p_0)$ for the toroidal fibre.

\section{Oka tools}
\label{sec:tools}

A dominating spray on a complex manifold $M$ is a holomorphic map $s:E\to M$ from a holomorphic vector bundle over $M$ such that $s(0_x)=x$ and $d s_{0_x}(E_x)=T_xM$ for every $x\in M$.
A manifold admitting such a spray is called \emph{Gromov elliptic}.
Gromov proved that a manifold admitting such a spray is Oka \cite{Gromov}.
Finitely many complete holomorphic vector fields spanning every tangent space give such a spray by composing their flows.
This is how the explicit vector fields below prove the required Oka assertions.

We also use the following results.
Here a Zariski-open set means the complement of a closed analytic subvariety.
\begin{enumerate}
\item \emph{Localization.} If every point has a Zariski-open Oka neighbourhood,
the manifold is Oka (Kusakabe \cite{Kusakabe}).
\item \emph{Coverings and bundles.} The Oka property is invariant in both directions under unramified holomorphic coverings \cite[Proposition~5.6.3]{ForstnericBook}.
For a holomorphic fibre bundle with Oka fibre,
the total space is Oka if and only if the base is Oka \cite[Corollary~1.3]{ForstnericOka2009};
see also \cite[Theorem~5.6.5]{ForstnericBook}.
\item \emph{Flexibility.} Smooth flexible affine varieties are Oka \cite{AFKKZ}.
The smooth locus of a suspension defined by a nonconstant regular function on a flexible affine variety is flexible \cite{AKZ}.
Removing a subvariety of codimension at least two from a flexible affine variety again gives a flexible quasi-affine variety \cite{FKZ}.
\item Products of Oka manifolds,
complex Lie groups and complex homogeneous manifolds are Oka \cite{Gromov,ForstnericBook}.
\end{enumerate}

\section{Unramified covers and the elliptic quotient}
\label{sec:covers}

Put $V_1=Y\setminus \Sigma_2$,
where $\Sigma_2$ is the reduced order-four multiple fibre.
On the base of $V_1$,
set
\[
q=\frac{t}{t-1},\qquad q=s^3.
\]
The only branch value retained in the base is $p_1$.
Locally at that point,
the normalization of the pullback is the space before quotienting by \eqref{eq:logaction}.
Proposition~\ref{prop:logfill} therefore identifies the induced map
\begin{equation}
\widehat V_1\longrightarrow V_1
\label{eq:cyclic}
\end{equation}
with a finite unramified three-sheeted covering.
The cusp has three inverse images,
at $s^3=1$.

Over the regular locus,
the monodromy-invariant subgroup $K$ defines the cover
\[
\C^2/\Pi K\longrightarrow\C^2/\Pi\Lambda.
\]
It extends over the point above $p_1$,
since the cubic base change kills the order-three monodromy.
It also extends over each cusp.
Indeed,
by \eqref{eq:Psi} the cusp filling is the infinite toric space modulo $\overline\Lambda$,
while
\[
K/\Lambda_{\rm tor}=\Z\overline{\widehat u},
\qquad B_0\overline{\widehat u}=\widehat w.
\]
Thus the required cusp cover is the quotient by the subgroup acting by one primitive horizontal translation of the fan.
A subgroup of the free proper action in Proposition~\ref{prop:cuspfill} is again free and proper.
We obtain an unramified covering
\begin{equation}
R_1\longrightarrow\widehat V_1.
\label{eq:Kcover}
\end{equation}

The three periods indexed by $K$ are
\[
(\tau,\mu),\qquad (1,0),\qquad(0,1).
\]
Projection to the first coordinate therefore gives a principal $\C^*$-bundle
\begin{equation}
\C^2/\Pi K\longrightarrow E_\tau=\C/(\Z+\tau\Z),
\label{eq:Cstar}
\end{equation}
whose fibre direction is $\widehat\delta$.
Since $\widehat\delta$ is fixed,
rather than merely preserved up to sign,
the structure group contains no inversion.

After the cubic base change,
the elliptic quotient extends to a smooth minimal elliptic surface $\Et\to\C_s$.
Its complement of the zero section $O$ is
\begin{equation}
\mathcal S=\Et\setminus O =\{Y^2=4X^3-3sX-1\}\subset\C^3.
\label{eq:S}
\end{equation}
This follows directly from the Weierstrass equation of Proposition~\ref{prop:weierstrass},
\[
y^2=4x^3-3t^3(t-1)x-t^4(t-1)^2
\]
by
\[
t=\frac{s^3}{s^3-1},\qquad x=\frac{s^4X}{(s^3-1)^2},\qquad y=\frac{s^6Y}{(s^3-1)^3}.
\]
The three singular fibres are irreducible $I_1$ fibres at $s^3=1$.

Put $A=Y-i$ and $Z=4X^2-3s$.
Equation \eqref{eq:S} becomes
\begin{equation}
XZ=A(A+2i).
\label{eq:dan}
\end{equation}
This is a flexible Danielewski surface.
It can also be written as $SL_2(\C)/T$:
on $G=\{ad-bc=2i\}$,
the invariants
\[
X=ab,\quad A=cb,\quad A+2i=ad,\quad Z=cd
\]
give the quotient by the torus that scales the two columns with opposite weights.
The section $P$ becomes
\[
P=(X,Y)=(0,i),
\]
which is disjoint from $O$.
Fibrewise translation by $-P$ extends to an automorphism of the relatively minimal rational elliptic surface and identifies $\Et\setminus P$ with $\mathcal S$ \cite{Karayayla}.

For the other multiple fibre,
use $(t-1)/t=s^4$,
so that
\[
t=\frac1{1-s^4},\qquad x=\frac{s^2X}{(1-s^4)^2},\qquad y=\frac{s^3Y}{(1-s^4)^3}.
\]
The same substitution in \eqref{eq:weierstrass} gives
\begin{equation}
\mathcal S_2=\{Y^2=4X^3-3X-s^2\}.
\label{eq:S2}
\end{equation}
With $A=Y-is$,
$B=Y+is$,
it is the flexible Danielewski surface
\[
AB=X(4X^2-3),
\]
and $(0,is)$ is a section disjoint from the zero section of its relatively minimal elliptic completion $\Et_2$.
The four singular fibres are irreducible $I_1$ fibres at $s^4=1$.

\begin{proposition}
\label{prop:outsideD}
The Zariski-open manifold $Y\setminus\Sing W$ is Oka,
where $W=f^{-1}(p_0)$.
\end{proposition}

\begin{proof}
First work on $Y\setminus(\Sigma_2\cup\Sing W)$ and use \eqref{eq:cyclic}--\eqref{eq:Kcover}.
At a cusp,
quotienting the source $A_2$-fan by $K/\Lambda_{\rm tor}$ identifies horizontal translates;
its central components are indexed by the remaining row coordinate.
After deleting $\Sing W$,
retain one row and delete every other row divisor.
Projection along the $\widehat\delta$-direction maps every cone of this one-row fan isomorphically onto the Tate fan.
Hence the resulting Zariski open is a principal $\C^*$-bundle over $\Et$ with the three nodal points removed.

That base is Oka.
The complements of $O$ and $P$ cover $\Et$;
after translation,
both are the flexible affine surface $\mathcal S$.
Removing the three nodes is a codimension-two removal,
so the two resulting quasi-affine surfaces are flexible by \cite{FKZ}.
They are Oka and cover the base.
Bundle invariance and localization make every one-row open Oka.
Such opens cover $R_1$,
and unramified descent gives
\[
Y\setminus(\Sigma_2\cup\Sing W)\quad\text{Oka}.
\]
For the order-four construction,
replace $(\Et,\mathcal S,P)$ by $(\Et_2,\mathcal S_2,(0,is))$ and the three nodes by the four nodes at $s^4=1$.
The same bundle and localization argument gives
\[
Y\setminus(\Sigma_1\cup\Sing W)\quad\text{Oka}.
\]
These two Zariski opens cover $Y\setminus\Sing W$,
so Kusakabe localization finishes the proof.
\end{proof}

\section{Small resolutions of the scalar model}
\label{sec:scalar}

Fix $a^3=1$.
In $\mathcal S$,
impose $uv=s-a$ and eliminate $s$.
The resulting threefold is
\begin{equation}
C_a=\{Y^2+3Xuv=4X^3-3aX-1\}\subset\C^4.
\label{eq:Ca}
\end{equation}
It has one ordinary double point at
\[
X_0=-\frac1{2a},\qquad Y=u=v=0.
\]
Indeed the partial derivatives of the defining equation vanish only at this point,
and its quadratic term there is a nondegenerate form in $X-X_0,Y,u,v$.

\begin{theorem}
\label{thm:scalar}
Every analytic small resolution $U_a\to C_a$ is Oka.
\end{theorem}

\begin{proof}
The contraction $C_a$ is the suspension $uv=s-a$ over the flexible surface $\mathcal S$.
Its smooth locus is flexible by \cite{AKZ},
hence Oka.
If $E\simeq\PP^1$ is the exceptional curve,
then
\begin{equation}
U_a\setminus E
\label{eq:UaE}
\end{equation}
is Oka.

Let $\mathcal B_a=\{X=1/a\}$;
it is disjoint from the node.
On the complement of $\mathcal B_a$,
take the unramified double cover
\[
r^2=X-1/a,\qquad r\in\C^*.
\]
Since $a^3=1$,
\begin{equation}
4X^3-3aX-1=4\left(X+\frac1{2a}\right)^2 \left(X-\frac1a\right).
\label{eq:factor}
\end{equation}
Put
\[
p=2(X-X_0)r-Y,\quad q=2(X-X_0)r+Y,\quad \alpha=3X,\quad h=4(X-X_0)r.
\]
Then the split cover is
\begin{equation}
pq=\alpha uv,\qquad p+q=h.
\label{eq:split}
\end{equation}
It has two nodes,
at the two zeros of $h$;
$\alpha$ is nonzero there.

Consider the kernel incidence resolution of
\[
M=\begin{pmatrix}p&\alpha u\\v&q\end{pmatrix}.
\]
In the projective kernel coordinate $[\xi:\eta]$,
the chart $\xi\ne0$,
with $\mu=\eta/\xi$,
is
\[
p=-\alpha u\mu,\qquad v=-h\mu-\alpha u\mu^2,
\]
and is therefore $\C^*\times\C^2$.
On $\eta\ne0$,
with $\lambda=\xi/\eta$,
it is
\begin{equation}
H=\{\alpha(r)u+\lambda h(r)+\lambda^2v=0\} \subset\C^*\times\C^3.
\label{eq:H}
\end{equation}
Let $r_+,r_-$ be the two zeros of $\alpha$,
put $h_i=h(r_i)\ne0$,
and define
\[
\mathcal D_i=\{r=r_i,\ h_i+\lambda v=0\}\subset H.
\]
The open $\{\lambda\ne0\}$ is $\C^*\times\C^*\times\C$,
by solving for $v$.
On $H_A=H\setminus(\mathcal D_+\cup \mathcal D_-)$,
the quotient $\ell=\lambda/\alpha$ extends holomorphically across $\alpha=0$,
because
\[
\ell=-\frac{u}{h+\lambda v}
\]
there.
Conversely,
\[
\lambda=\alpha \ell,\qquad u=-\ell h-\alpha \ell^2v
\]
identify $H_A$ with $\C^*\times\C^2$.
These two opens cover $H$.
Together with the first incidence chart they form a Zariski-open Oka cover of the kernel resolution.
Transposition gives the same conclusion for the cokernel resolution.

We also require these resolutions after deleting either exceptional curve.
If $\rho_j$ is a zero of $h$,
deletion replaces one factor $\C^*\times\C$ in each chart by
\begin{equation}
Q_J=(\C^*\times\C)\setminus \{(\rho_j,0):j\in J\}.
\label{eq:QJ}
\end{equation}
This surface is elliptic.
Indeed,
with $p_J(r)=\prod_{j\in J}(r-\rho_j)$,
consider the complete fields
\[
ur\frac{\partial}{\partial r},\qquad p_J(r)\frac{\partial}{\partial u},\qquad u\frac{\partial}{\partial u},
\]
and the conjugate of the first field by the automorphism $(r,u)\mapsto(r,u+p_J(r))$.
All fields preserve the deleted set.
At $u\ne0$,
the first and third span;
at $u=0$ in $Q_J$,
one has $p_J(r)\ne0$,
so the second field and the conjugated field span.
Their flows therefore give a dominating spray.
Products and localization show that a global kernel or cokernel resolution remains Oka after deleting any chosen subset of its two exceptional curves.

It remains to identify the ruling pulled back from $U_a$.
We use the standard local classification of a threefold ordinary double point:
its analytic small resolutions are the two incidence resolutions,
distinguished by the two rulings \cite{Friedman}.
The deck involution satisfies
\[
r\mapsto-r,\qquad p\mapsto-q,\qquad q\mapsto-p,\qquad h\mapsto-h.
\]
If $\mathsf J=\begin{psmallmatrix}0&1\\1&0\end{psmallmatrix}$ and $\mathsf D=\operatorname{diag}(-1,1)$,
then the transformed matrix is
\[
-\mathsf D\,\mathsf J M^{\mathsf t}\mathsf J\,\mathsf D.
\]
Thus the involution exchanges the kernel and cokernel rulings.
The pullback of either small resolution of the original node consequently has opposite rulings at the two lifted nodes.
Removing one lifted exceptional curve makes it a global kernel resolution with that curve deleted;
removing the other makes it a global cokernel resolution with that curve deleted.
These two Oka Zariski opens cover the mixed resolution.
Unramified descent proves that $U_a\setminus \mathcal B_a$ is Oka.

Finally,
$\mathcal B_a\cap E=\varnothing$,
so $U_a\setminus \mathcal B_a$ and \eqref{eq:UaE} form a Zariski-open Oka cover of $U_a$.
Localization proves the theorem.
\end{proof}

\section{Line-bundle twisting}
\label{sec:twisted}

\begin{theorem}
\label{thm:twisted}
Let $L\to\Et$ be any holomorphic line bundle and let $a^3=1$.
In the total space of $L\oplus L^{-1}$,
put
\begin{equation}
Z_a(L)=\{uv=s-a\}.
\label{eq:ZaL}
\end{equation}
Every analytic small resolution $R_a(L)\to Z_a(L)$ is Oka.
\end{theorem}

\begin{proof}
First restrict to $\mathcal S=\Et\setminus O$.
Write $L^\times$ for the total space of $L$ with its zero section removed,
and pull back along the principal bundle $L^\times\to\mathcal S$.
Its tautological section trivializes both pulled-back line bundles,
so the pulled contraction is the fibre product of the scalar model $C_a$ with $L^\times\to\mathcal S$.

The singular locus upstairs is the connected $\C^*$-fibre over the node.
The two possible small-resolution rulings form a discrete set,
while the structure group $\C^*$ is connected and acts by opposite scalar multiplication on $u,v$;
it preserves each ruling.
Hence the pulled-back small resolution is the fibre product of one of the two scalar small resolutions with $L^\times$.
It is a principal $\C^*$-bundle over the scalar resolution and is Oka by Theorem~\ref{thm:scalar}.
Since it is also a principal $\C^*$-bundle over $R_a(L)|_{\mathcal S}$,
bundle invariance implies that $R_a(L)|_{\mathcal S}$ is Oka.

The two opens $\Et\setminus O$ and $\Et\setminus P$ cover $\Et$.
Translation by $-P$ identifies the second with $\mathcal S$,
preserves the base coordinate $s$,
and carries $L$ to another holomorphic line bundle.
The preceding argument therefore makes the inverse images of both base opens Oka.
They form a Zariski-open cover of $R_a(L)$,
and localization completes the proof.
\end{proof}

\section{Two adjacent toroidal rows}
\label{sec:rows}

We now reconstruct from \eqref{eq:Psi} the strata omitted in Proposition~\ref{prop:outsideD}.
The source cusp fan is the cone over the standard $A_2$-triangulation of $\R^2$,
with vertices $\Z^2\times\{1\}$.
After quotienting by $K/\Lambda_{\rm tor}$,
horizontal translates are identified and the remaining components are indexed by rows.

Take two adjacent rows.
One fundamental square has rays
\[
r_{00}=(0,0,1),\ r_{10}=(1,0,1),\ r_{01}=(0,1,1),\ r_{11}=(1,1,1),
\]
and is triangulated along one diagonal.
On the dense torus define
\begin{equation}
\xi=x_1,\qquad \eta=t_c/x_1,\qquad u=x_2,\qquad v=t_c/x_2.
\label{eq:fourchars}
\end{equation}
These are the four primitive generators of the dual semigroup of the untriangulated cone,
and
\begin{equation}
\xi\eta=uv=t_c.
\label{eq:conifold}
\end{equation}
Thus forgetting the diagonal contracts the two-row fan to a conifold,
and the source diagonal is one of its small resolutions.

\begin{lemma}
\label{lem:strip}
Use the orientation determined by the fixed vector $\widehat\delta$.
Choose a lower row at each of the three lifted cusps of $R_1$,
and let $L^\times\to\Et$ be the resulting one-row principal $\C^*$-bundle.
Adding the upper adjacent row and all intervening cones gives a small resolution of
\begin{equation}
Z(L)=\{uv=s^3-1\}\subset\Tot(L\oplus L^{-1}).
\label{eq:ZL}
\end{equation}
The ruling at each of the three nodes is the ruling prescribed by the local source fan.
\end{lemma}

\begin{proof}
Projection of a one-row fan along $\widehat\delta$ is an isomorphism on every cone onto the Tate fan.
Hence the one-row open is a principal $\C^*$-bundle over $\Et$,
including the nodal points,
and therefore is the punctured total space of a line bundle $L$.

Locally at a cusp,
\eqref{eq:fourchars}--\eqref{eq:conifold} give the desired contraction.
On an overlap the fixed primitive vector $\widehat\delta$ prevents inversion,
so a change of fibre coordinate has the form $u'=g u$,
with $g$ holomorphic and nowhere zero.
Normalize the local cusp parameters by the global function $s^3-1$;
the equation then forces $v'=g^{-1}v$.
Formula \eqref{eq:Psi} supplies precisely these multipliers.
The contractions therefore glue to \eqref{eq:ZL}.

At a nodal point choose coordinates $x,y$ on $\Et$ with $s^3-1=xy$.
The contraction is $uv=xy$,
and the two triangles in the square are the two standard affine charts of one of its small resolutions.
This proves both smoothness and the assertion about the ruling.
\end{proof}

Let $a_1,a_2,a_3$ be the cube roots of unity and,
in \eqref{eq:ZL},
put
\[
F_k=\{u=0,\ s=a_k\}.
\]
For fixed $i$,
delete from the small resolution the proper transforms of $F_k$ for $k\ne i$.
On the resulting open,
the rational section
\begin{equation}
v_i=\frac{v}{\prod_{k\ne i}(s-a_k)}
\label{eq:divide}
\end{equation}
is holomorphic.
Indeed,
over $s=a_k$ the deleted component is $u=0$,
while on its complement the equation supplies the cancellation.
The new equation is
\begin{equation}
uv_i=s-a_i.
\label{eq:onebranch}
\end{equation}
Near $uv=xy$,
the divisor $F_k$ has branches $\{u=x=0\}$ and $\{u=y=0\}$.
For either ruling,
the proper transform of their union contains the exceptional curve.
For example,
in a kernel chart one has $u=-x\lambda$,
$y=-v\lambda$;
the transform of $u=0$ is the union $\{x=0\}\cup\{\lambda=0\}$,
and the exceptional curve lies in $\{x=0\}$.
The other ruling is identical after transposition.
Thus the deleted exceptional curves over $a_k$,
$k\ne i$,
disappear exactly as required for \eqref{eq:divide} to identify the open with a small resolution of the one-node model \eqref{eq:onebranch}.

\begin{proposition}
\label{prop:stripOka}
Every simultaneous two-row strip in $R_1$ is Oka.
\end{proposition}

\begin{proof}
The three opens obtained from \eqref{eq:divide} are Oka by Theorem~\ref{thm:twisted}.
They cover:
away from the exceptional set a point lies on at most one $F_k$,
and the exceptional curve over $a_i$ is retained in the $i$-th open.
Kusakabe localization proves the proposition.
\end{proof}

\section{Proof of the Oka property}
\label{sec:global}

\begin{proof}[Proof of Theorem~\ref{thm:main}]
The normalization \eqref{eq:normbeta} gives $D(\beta^*)\le-1$,
so $c_0=-i$ is admissible.
At every lifted cusp choose two adjacent rows and delete all other row divisors.
The deleted family is locally finite,
so its complement is a Zariski-open simultaneous two-row strip.
These strips cover $R_1$:
every triangle of the $A_2$-fan has vertices in two adjacent rows.
Proposition~\ref{prop:stripOka} and localization therefore show that $R_1$ is Oka.

Both arrows
\[
R_1\longrightarrow\widehat V_1 \longrightarrow V_1=Y\setminus \Sigma_2
\]
are unramified holomorphic coverings.
Covering invariance gives that $Y\setminus \Sigma_2$ is Oka.
This open contains $\Sing W$.
Proposition~\ref{prop:outsideD} gives the second Oka open $Y\setminus\Sing W$,
which contains $\Sigma_2$.
Hence
\[
Y=(Y\setminus \Sigma_2)\cup(Y\setminus\Sing W)
\]
is a Zariski-open Oka cover.
Kusakabe localization proves the Oka assertion for $Y$.
\end{proof}

\section{The period parameter and the sphere implication}
\label{sec:parameter}

\begin{proof}[Proof of Theorem~\ref{thm:family}]
Fix an admissible $c_0$.
Proposition~\ref{prop:compact} supplies the compact threefold $Y_{c_0}$.
Changing $c_0$ changes only $\beta$ in the period matrix;
the periods on $K=\ker\gamma$,
the monodromy,
the twists and the toroidal fan are unchanged.
At the cusp the change adds a constant to $\beta+\tau$,
so the multipliers in \eqref{eq:Psi} remain holomorphic units.
Thus the unramified covers and the row and node models used above apply to $Y_{c_0}$.
The associated line bundle may vary,
but Theorem~\ref{thm:twisted} applies to every holomorphic line bundle.
Consequently Proposition~\ref{prop:outsideD} and Section~\ref{sec:global} give the same two Zariski-open Oka complements,
and localization proves that $Y_{c_0}$ is Oka.
\end{proof}

The sphere consequence is separate from these Oka proofs.
Under Condition~\ref{cond:sphere},
each $Y_{c_0}$ is diffeomorphic to $S^6$,
as asserted in \cite[Remark~6.4 and Theorem~8.1]{Alpoge} and deduced in \cite[Corollary~2.30]{DGKX}.
Together with Theorem~\ref{thm:family},
transporting these complex structures along diffeomorphisms gives Oka complex structures on $S^6$.

\section{Further Oka properties and questions}
\label{sec:questions}

Du--Guo--Kusakabe--Xie establish several further Oka properties of the family.
Besides the relative Oka-map assertion in \cite[Theorem~1.1]{DGKX},
they prove that deleting or simultaneously blowing up any finite set of distinct points preserves the Oka property \cite[Theorem~1.2]{DGKX}.
Every proper holomorphic deformation family with a marked fibre isomorphic to $Y_{c_0}$ is also an Oka map near that fibre \cite[Corollary~7.16]{DGKX}.
These results are entirely due to them.

Forstneri\v c--L\'arusson prove that every projective Oka manifold is elliptic \cite[Theorem~1.1]{ForstnericLarusson}.
Their theorem does not apply here,
since $a(Y_{c_0})=1<3$ \cite[Theorem~9.1(i)]{Alpoge};
the general compact nonprojective case remains open \cite[Problem~1.2]{ForstnericLarusson}.
We therefore retain the following question from the first version;
it does not assume Condition~\ref{cond:sphere}.

\begin{question}
\label{q:ellipticity}
Is $Y_{c_0}$ Gromov elliptic for every $c_0$ with $\operatorname{Im}c_0<0$?
\end{question}

Fix an admissible $c_0$ and write $Y=Y_{c_0}$.
The final cover in our proof is
\[
V_1=Y\setminus \Sigma_2,\qquad V_2=Y\setminus\Sing W.
\]
We prove that both opens are Oka,
but do not construct a dominating spray on a vector bundle over either entire open.
The explicit sprays occur on surface and incidence models;
passing to the twisted models and row covers uses bundle invariance and repeated localization.
Descending a spray through the unramified covers would additionally require compatible descent data for its vector bundle and spray map.
The covering invariance of the Oka property does not supply these data.
Likewise,
the sprays in Du--Guo--Kusakabe--Xie's relative argument \cite[Section~3.4]{DGKX} have holomorphic maps from convex neighbourhoods as their cores.
They satisfy Kusakabe's criterion for Oka maps \cite[Theorem~1.3]{KusakabeMaps},
without furnishing sprays over the identity maps of the $V_i$.

Even if dominating sprays on both $V_i$ were constructed,
a further extension or gluing argument would be needed to obtain one on a vector bundle over all of $Y$.
The spray bundles and maps need not agree on overlaps,
and a smooth partition of unity does not preserve holomorphicity.
No analogous analytic-Zariski localization theorem is presently available for holomorphic ellipticity or subellipticity \cite[Introduction]{ForstnericLarusson}.
In particular,
sprays on different open sets do not establish subellipticity,
which requires finitely many sprays defined over the whole manifold whose fibre derivatives jointly span its tangent spaces.

There are concrete obstructions to using trivial spray bundles.
Du--Guo--Kusakabe--Xie prove \cite[Theorem~1.3]{DGKX} that
\[
h^0(Y,T_Y)=1.
\]
If $s:Y\times\C^N\to Y$ were a dominating spray,
its parameter derivatives at zero would be global holomorphic vector fields spanning every three-dimensional tangent space,
contrary to this identity;
compare \cite[Proposition~6.2]{ForstnericArakelian}.
The same obstruction holds on $V_2$.
The normal-crossing description of $W$ \cite[Theorem~2.21]{DGKX} implies that $\Sing W$ has complex codimension two in $Y$.
Hartogs extension applied to the local coefficients of vector fields therefore gives
\[
H^0(V_2,TV_2)\cong H^0(Y,T_Y),
\]
so $V_2$ also admits no dominating spray on a trivial bundle.

For $V_1$ there is a direct obstruction from the varying compact torus fibres.
Over a simply connected disc in $B^\circ$,
use the period matrix $\Pi(t)=[Z(t),I_2]$ of \eqref{eq:Pi} with the chosen parameter.
Lift a holomorphic vector field on $V_1$ to
$a(t)\partial_t+b(t,\zeta)\partial_\zeta$ on the vector-space cover of this torus family;
compactness of the fibres makes its horizontal coefficient independent of $\zeta$.
Lattice equivariance gives
\[
b(t,\zeta+\Pi(t)\lambda)
=b(t,\zeta)+a(t)\Pi'(t)\lambda,
\qquad \lambda\in\Z^4.
\]
Thus $\partial_\zeta b$ is lattice-periodic and hence constant in $\zeta$.
The two constant period columns force this matrix to vanish,
and the displayed identity then gives $a(t)Z'(t)=0$.
Since $\tau$ is nonconstant,
$a=0$.
Every global holomorphic vector field on $V_1$ is therefore tangent to its regular torus fibres,
so parameter derivatives of a spray on $V_1\times\C^N$ cannot span the full tangent space.
These arguments exclude trivial spray bundles on $Y$, $V_1$ and $V_2$;
they do not exclude dominating sprays on nontrivial holomorphic vector bundles.

Under Condition~\ref{cond:sphere} and to the author's knowledge,
the family $Y_{c_0}$ with $\operatorname{Im}c_0<0$ represents all currently known complex structures on $S^6$.
They are all Oka by Theorem~\ref{thm:family},
which does not use this condition.
The equivalence criterion \eqref{eq:classification} is part of Theorem~\ref{thm:DGKX},
due to Du--Guo--Kusakabe--Xie;
it classifies members of this family and makes no assertion that the family exhausts all complex structures on $S^6$.

\begin{question}
\label{q:other-structures}
Under Condition~\ref{cond:sphere},
does $S^6$ admit a complex structure not biholomorphic to any $Y_{c_0}$ with $\operatorname{Im}c_0<0$?
If such structures exist,
must they be Oka?
\end{question}

\section*{Acknowledgements}
The author thanks Franc Forstneri\v c for his remarks and advice concerning the exposition of Oka theory,
and Song-Yan Xie for correspondence about the independent work \cite{DGKX}.

\section*{Declarations}

\noindent \textbf{Use of generative artificial intelligence.} OpenAI's Codex was used as an interactive assistant in developing and checking the arguments,
performing calculations,
locating references,
and drafting and revising the manuscript.
The author retains responsibility for the mathematical content and any errors.

\end{document}